\documentclass[11pt,reqno]{article}
\usepackage{amsgen, amsmath, amsfonts, amsthm, amssymb, enumerate,amscd,graphics,dsfont,comment,tikz-cd}

\newtheorem{defn}{Definition}[section]
\newtheorem{prop}[defn]{Proposition}

\newtheorem{thm}[defn]{Theorem}

\newtheorem{rem}[defn]{Remark}

\newtheorem {const}[defn]{Construction}

\newcommand {\OP}{{O^{+}}}
\newcommand {\OT}{{\tilde{O}}}
\newcommand {\OPT}{{\tilde{O}^{+}}}

\newcommand {\hL}{{\hat{L}}}
\newcommand {\DD}{{\mathcal D}}

\title{Indecomposable motivic cycles on  $K3$ surfaces of degree $2$}
\author{Ramesh Sreekantan}

\newcommand {\ZZ}{{\mathds Z}}

\newcommand {\XX}{{\mathcal X}}

\newcommand {\C}{{\mathds C}}

\newcommand {\CC}{{\mathcal C}}
\newcommand {\Q}{{\mathds Q}}

\newcommand {\OO}{{\mathcal O}}

\newcommand {\HH}{{\mathds H}}
\newcommand {\UH}{{\mathcal H}}
\newcommand {\LL}{{\mathcal L}}

\newcommand {\M}{{\mathcal M}}

\newcommand {\CP}{{\mathds P}}

\newcommand {\D}{{\mathcal D}}

\newcommand {\TCC}{\tilde{\mathcal C}}

\newcommand {\CPP}{{\CP \times \CP}}

\newcommand{\TC}{\tilde{C}}

\newcommand{\BX}{\bar{X}}

\newcommand {\KP}{{{\mathcal P}}}

\def\Ker{\operatorname{Ker}}

\def\div{\operatorname{div}}
\def\deg{\operatorname{deg}}
\def\mod{\operatorname{mod}}
\def\dim{\operatorname{dim}}
\def\Spec{\operatorname{Spec}}

\def\ord{\operatorname{ord}}

\author{Ramesh Sreekantan}

\begin{document}
	\baselineskip=17pt
	
	\maketitle

	\begin{abstract}
		
		In this paper we construct new indecomposable motivic cycles in the group $H^3_{\M}(X,\Q(2))$ where $X$ is a degree $2$ $K3$ surface. This generalizes our  construction in \cite{sree-simple,sree-split} for Kummer surfaces of Abelian surfaces as well as the recent work of Ma and Sato \cite{masa} on degree $2$ $K3$ surfaces. 
		
	\end{abstract}


\section{Introduction}
	
	\subsection{Degree 2 $K3$ surfaces}
	
A degree $2$ $K3$ surface $X$ is a $K3$ surface on which there is a line bundle $\LL$ such that $\LL^2=(\LL.\LL)=2$. This is equivalent to the $K3$ surface being a double cover of $\CP^2$ ramified at a sextic. The linear system $|\LL|$ gives a map to $\CP^2$ which is ramified at a sextic. 

Conversely if $S$ is a sextic  in $\CP^2$ with at most isolated rational double point singularities, then the desingularisation of the double cover ramified at $S$ is a $K3$ surface of degree $2$. 

Generically the rank of the Neron-Severi group (the Picard number $\rho(X)$) is $1$ - though there are some interesting special cases when the rank is much larger. If the sextic is a product of six lines, for instance,  then the rank is $16$ and the moduli space of such $K3$ surfaces is four dimensional. 

On this $4$-fold there are two notable subvarieties. If the six lines are futher tangent to a conic, then the $K3$ surface is a Kummer surface of a principally polarised Abelian surface and the Picard number is $17$ \cite{biwi}.  If the six lines degenerate such that three of them are concurrent then Clingher and Malmendier show that the $K3$ surface is the Kummer surface of an Abelian surface with a polarization of type $(1,2)$ \cite{clma}. Other cases have been studied by Yu and Zheng \cite{yuzh}.

\subsection{Motivic cycles}

In this paper we construct elements in the motivic cohomology group $H^3_{\M}(X,\Q(2))$ where $X$ is a degree $2$ $K3$ surface. The group $H^3_{\M}(X,\Q(2))$ has many avatars - it is usually defined as a graded piece of $K_1(X)$ with respect to the Adams filtration, but for us we will use the fact that it is the same as the Bloch higher Chow group $CH^2(X,1)$.

The cycles we construct are generically {\em indecomposable} -- namely they are non-trivial elements of the quotient of the motivic cohomology group by subgroup of cycles coming from lower graded pieces of the filtration. In a sense they are `new' cycles.

In general it is not so easy to find indecomposable cycles and the discovery of them is a minor cause for celebration. They have a number of applications - from algebraicity of values of Greens functions \cite{mell},\cite{sree-simple},\cite{sree-split}, \cite{kerr} to the Hodge-$\D$-conjecture \cite{chle}, \cite{sree2008},\cite{sree2014} to torsion in co-dimension $2$ \cite{mild},\cite{spie}. For varieties over number fields, the existence of indecomposable cycles is predicted by the Beilinson conjectures though there are only a few cases when these have been realised.

The method of construction is similar to that in \cite{sree-simple} where we studied the case of  Abelian surfaces in some detail. The idea is to show the existence of certain rational curves on $\CP^2$. In this paper we show that the same construction can be generalised to this much larger class of varieties, The Kummer surfaces of principally polarised Abelian surfaces are special cases of $K3$ surfaces of degree $2$.

As this paper was nearing completion we became aware of the preprint of Ma and Sato \cite{masa} where they also construct indecomposable motovic cycles in this case. Their methods of proof are quite different though. Their construction can be viewed as a special case when the rational curve is a line.

\subsection{Acknowledgements}

The history of this article is linked with the online seminar on $K3$ surfaces which has resulted in these proceedings.  I spoke in the seminar on the construction in the article \cite{sree-simple} which was about Kummer surfaces of principally polarised Abelian surfaces, Prof. Armand Brumer was in the audience and he asked about the non-principally polarised case. When I dug into the literature I became aware of the work of Clingher and Malmendier \cite{clma} on the $(1,2)$ case and it was clear that the argument extends to that case. It soon became clear that the argument extends to the general degree $2$ case as well. So I would first like to thank Devendra Tiwari and the rest of the organisers for inviting me to speak and second would like to thank Armand Brumer for his comment. I would also like to thank Shouhei Ma, Ken Sato and Subham Sarkar for their comments. Tom Graber and Ritwik Mukherjee for their help with questions on enumerative geometry and  Arvind Nair for his remarks on orthogonal Shimura varieties. I would also like to thank the referee for their numerous comments and suggestions.  Finally I would like to thank the Indian Statistical Institute for their support.

\section{$K3$ surfaces of degree 2}

\subsection{$K3$ surfaces as double covers of $\CP^2$ ramified at a sextic}

Let  $X$ be a $K3$ surface with a line bundle $\LL$ such that $\LL^2=2$. The linear system $|\LL|$ has no fixed points and maps to $\CP^2$. This is a double cover ramified at a sextic $S=S_X$. We call $S_X$ the {\em associated sextic} of $X$. The hyperplane section on $\CP^2$ pulls back to a double cover of a line ramified at six points - hence is a genus $2$ curve in $|\LL|$ \cite{maye}. On the complement of an infinite union of closed subvarieties,  the Picard number $\rho(X)$  of  such $K3$ surfaces is $1$ generated by this curve \cite{shah}. Such $K3$ surfaces are called {\em generic}.

Conversely suppose $S$ is a sextic in $\CP^2$  with only isolated double point singularities. Then the minimal desingularisation of the double cover ramified at $S$ is a $K3$ surface of degree $2$. If there are  double points on the sextic the Picard number of the   $K3$ surfaces is larger as the exceptional cycles in the desingularization give new elements of the Neron-Severi group. 

The compactification of the moduli space of smooth sextics was first studied by Shah \cite{shah} and more recently Yu and Zheng \cite{yuzh} have studied the singular sextic cases.

\subsection{Lattices and the Moduli of $K3$ surfaces}

In this section we state some results about the moduli space of $K3$ surfaces of degree $2$. These results can be found in Bruinier \cite{brun123}. We have followed Peterson \cite{pete}.

If $X$ is a $K3$ surface, the group $H^2(X,\ZZ)$  is of rank $22$. As a lattice it is 
$$H^2(X,\ZZ)=U^{\oplus 3} \bigoplus E_8(-1)^{\oplus 2}$$
where $U$ is a unimodular lattice of rank $2$ and signature $(1,1)$ and $E_8(-1)$ is the root lattice of the Lie algebra $E_8$ with the $(-1)$ indicating that we take the negative of the standard positive definite quadratic form. Hence this is a lattice of signature $(3,19)$ 

In $H^2(X,\ZZ)$ the class of $\LL$ gives us an element $\LL$ of square $2$. Let $L_2$ be the orthogonal complement $\LL^{\perp}$ of $\LL$ in $H^2(X,\ZZ)$. This is a lattice of signature $(2,19)$  
$$L_2=<-2>\bigoplus U^{\oplus 2} \bigoplus E_8(-1)^{\oplus 3}$$
where $<-2>$ is the rank 1 lattice generated by an element $w$ such that $(w,w)=-2$. 

The moduli of $K3$ surfaces of degree $2$ can be realised as an orthogonal Shimura variety associated to the lattice $L_2$ as follows.

For a lattice $L$ let $O(L)$ denote the group of isometries of $L$. Let $\hL$ denote the dual lattice 
$$\hL=\{m \in L \otimes \Q|(m,n) \in \ZZ \text{ for all } n \in L\}$$
The discriminant group is $D_L=\hL/L$. Let $\OT(\hL)=\Ker\{O(L) \rightarrow O(D_L)\}$. Let $\OP(L)$ denote the subgroup of $O(L)$ of spinor norm $1$ and $\OPT(L)=\OT(L) \cap \OP(L)$. 
		
Let 
$$\DD \cup \overline{\DD}=\{\C z|\;(z,z)=0,(z,\bar{z})>0\} \subset \CP(L \otimes \C)$$
Choose a component $\DD$. The moduli space of $K3$ surfaces of degree $2$ with transcendental lattice $L$ is 
$$\M_L=\OPT(L) \backslash \DD$$
An alternative description is as a quotient of the Hermitian symmetric space of $O(L)$. If $L$ is of signature $(2,n)$ this is 

$$\HH=\frac{O(2,n)}{O(2) \times O(n)}$$
The moduli space $\M$ is 
$$\M=\OPT(L) \backslash \HH$$
When $L$ is the lattice $L_2$ this moduli space $\M_2$ is a $19$ dimensional variety. If one considers singular sextics, the transcendental lattice becomes smaller as the exceptional fibres are elements of the Neron-Severi group. One interesting submoduli is when the sextic $S$ degenerates in to six lines. Here the Picard number is $16$ as the Neron-Severi is generated by the class of $\LL$ and the $15$ exceptional fibres over the nodes of $S$. The transcendental lattice is hence of signature $(2,4)$ and the moduli space is $4$ dimensional. 

In this there is a particularly interesting divisor -- if the six lines are tangent to a conic. Such  $K3$ surfaces have Picard number $17$ and the transcendental lattice is of signature $(2,3)$. The new cycle is obtained as a component of the double cover of the conic. 
These $K3$ surfaces are the Kummer surfaces of principaly polarised abelian surfaces  and the moduli is nothing but an alternative description of the Siegel modular threefold.   

There are several other moduli corresponding to nodal sextics. These are described in Yu and Zheng. \cite{yuzh}. Some of them correspond to $K3$ double covers of del Pezzo surfaces.

\subsection{$K3$ surfaces with larger Picard number}

As mentioned above, when the sextic has nodal singulaties, the corresponding $K3$ surface has additional elements in the Neron-Severi group coming from the exceptional fibres. In this section we will determine another condition under which there are new elements in the Neron-Severi group. In what follows we use $\LL$ to denote the class of the line bundle $\LL$ in the Neron-Severi group. Let $D$ be an element of $NS(X)$ which is not a multiple of $\LL$. The moduli of those $K3$  surfaces of degree $2$ where the Neron-Severi group is generated by $\LL$ and $D$ is a divisor on $\M_2$. 

Recall that $(\LL,\LL)=2$. Consider the intersection matrix of the lattice generated by $\LL$ and $D$, 
$$ A=\begin{pmatrix} 2 & (\LL.D) \\ (\LL.D) & (D.D) \end{pmatrix}$$
Let $\Delta=\Delta(D)=-\det(A)=(\LL.D)^2-2D^2$ and $\delta=(D.\LL) \mod 2$, so $\delta=\{0,1\}$. From the Hodge Index theorem we have that $\Delta(D) \geq 0$ and $\Delta >0$ if and only if  $D$ is not a multiple of $\LL$.

The {\em irreducible Noether-Lefschetz divisor} $\KP_{\Delta,\delta}$ is the closure of the set of polarised $K3$ surfaces where the Picard lattice is a rank $2$ lattice with discriminant $\Delta(D)=\Delta$ and class $\delta$. 

For number theoretic applications one considers Heegner divisors which are sums of these $\KP_{\Delta,\delta}$ for certain choices of $\Delta$. \cite{pete}. Numbers determined by them  are known to appear as coefficients of modular forms.

It is known that the Picard number of a Abelian surface with real multiplication  is at least  $2$ which translates to the corresponding Kummer surface having Picard number at least $18$. Birkenhake-Wilhelm \cite{biwi} showed, generalizing Humbert, that this corresponds to certain special rational curves on the associated $\CP^2$. Conversely, if one has special rational curves on $\CP^2$ there are additional cycles in the corresponding $K3$ surface. 

Along the lines of their argument, we have the following. 


\begin{prop}
	Let $X$ be a  $K3$ surface of degree $2$ and $S_X$ the associated sextic in $\CP^2$. Let $\pi:X \rightarrow \CP^2$ denote the double cover of $\CP^2$ ramified at $S_X$ induced by $\LL$.  Suppose there exists a rational curve  $Q$ of degree $d \neq 6$ in $\CP^2$ such that $Q$ meets $S_X$ only at points of {\em even} multiplicity - including possibly nodes of $S_X$. 
	
	 Then $Q$ determines an element in the Neron-Severi group of $X$ whose class is not a multiple of the class of $\LL$. In particular, the moduli point of $X$ lies on $\KP_{\Delta,\delta}$ for some $\Delta$ and $\delta$. 
	
\label{specialcurve}
\end{prop}

\begin{proof} The Picard number of $\CP^2$ is $1$ so the class of $Q$ is  a multiple of the class of the hyperplane section $H$. Since it is of degree $d$, $Q \in |dH|$. $\pi^*(H)=\LL$ hence $\pi^{-1}(Q) \in |\LL^d|$. In particular, it is not a new element. So we cannot simply use the preimage of $Q$. 
	
Let $C=\pi^{-1}(Q)$ be the double cover of $Q$ induced by the map $\pi:X \rightarrow \CP^2$. Let $\eta:\TC \longrightarrow C$ denote its normalization. The ramification points of $\pi:C \longrightarrow Q$ are the points $S_X \cap Q$. By assumption $S_X$ and $Q$ meet with even multiplicity. In  the double cover $C$ these are nodal points.  In the normalization $\TC$, these are no longer ramified. Since $Q$ is rational the normalization $\TC$ is then  an {\em unramified} double cover of $\CP^1$. This cannot be irreducible as $\CP^1$ is simply connected and does not have an irreducible unramified double cover. Since it is a double cover it has two components, say $\TC_1$ and $\TC_2$. Hence $C=C_1 \cup C_2$ for some possibly singular rational curves $C_1$ and $C_2$. 

We claim the class of $C_1$ (or $C_2$) is not  a multiple of the class of $\LL$. To do this we have to compute $\Delta(C_1)$. Recall that if $C_1$ is a multiple of $\LL$ then $\Delta(C_1)=0$. 

Let $\iota$ be the involution determined by the double cover. Then $\iota(C_1)=C_2$. On the other hand $\iota$ acts trivially on the class of $\LL$. If the class of $C_1$ is a multiple of the class of $\LL$ then $C_1$ and $C_2$ are equivalent. Then 
$$C_1^2=(C_1.C_2)=3d$$
as $C_1$ and $C_2$ meet at the $3d$ points of intersection lying over $Q \cap S_X$.

Further $C_1+C_2=\pi^*(dH)=\LL^d$. So 
$$(C_1.\LL)=\frac{1}{d}(C_1.(C_1+C_2))=\frac{1}{d}(2 C_1^2)=6$$
This gives $\Delta(C_1)=2C_1^2-(C_1.\LL)^2=6d-36$. On then other hand, it is $0$. This implies  $d=6$. However, we have assumed $d \neq 6$.

\end{proof}

In fact the theorem likely holds for higher genus curves as well \cite{biwi}, Proposition 6.4. We will only be concerned with the case when the curve is rational in what follows and hence will restrict ourselves to that case. 

The converse need not be true. Namely there exists $\KP_{\Delta,\delta}$ such that the extra cycle need not be represented by a {\em rational} curve. 

A rational curve $Q$ which meets the sextic $S$ at points of even multiplicity determines some nodes and orders of tangency, Let $\KP_Q$ denote the component of  $\KP_{\Delta(Q),\delta}$ where, for a $K3$ surface $X$ corresponding to a point on $\KP_{\Delta(Q),\delta}$, there is a rational curve $Q_X$ which meets $S_X$  is exactly the same manner.

If $S_X$ is smooth then the only possibility is that $Q_X$ is tangent to $S_X$ at some points. However, if $S_X$ has nodes and components, $P_{\Delta(D),\delta}$ could have more components. For instance, as mentioned above, nodal sextics correspond to interesting submoduli of the moduli of degree $2$ $K3$ surfaces. In the case when the sextic is a union of six lines $\ell_i$  tangent to a conic it is known that the $K3$ surface $X$ is the Kummer surface of a principally polarized  Abelian surface. Here the sextic has $15$ nodes $q_{ij}=\ell_i \cap \ell_j$  corresponding to the points of intersection of the lines. In this case $\Delta$ is called the {\em Humbert invariant}  and corresponds to the moduli of Abelian surfaces with endomorphism ring containing $\ZZ[\sqrt{\Delta}]$. For a generic Kummer surface the Picard number is $17$ and on these Humbert surfaces the Picard number is at least $18$. 

As an example, if there is a second conic passing thought $5$ of the $15$ nodes - hence it meets five of the lines at these points and suppose it is tangent to the remaining line, then $\Delta=5$ and $\KP_{\Delta(Q),\delta}$ is the moduli of Kummer surfaces of Abelian surfaces with real multiplication by $\ZZ({\sqrt 5})$. The conics passing through the five points $q_{12},q_{23},q_{34},q_{45}$ and $q_{51}$ and tangent to $\ell_6$ determine a component of $\KP_{\Delta(Q),\delta}$ and similarly, the conics passing through $q_{12},q_{23},q_{34},q_{46},q_{61}$ and tangent to $\ell_5$ will determine a different component. This special  case of Kummer surfaces of simple Abelian surfaces  was studied in some detail in \cite{sree-simple}.

In the case when the sextic degenerates to six lines three of which are concurrent the $K3$ surface is the Kummer surface of an Abelian surface with polarization of type $(1,2)$ \cite{clma}. Once again the $\KP_{\Delta,\delta}$ will correspond to Abelian surfaces with extra endomorphisms.

	\section{Motivic cycles} 
	\subsection{Presentation of cycles in $H^3_{\M}(X,\Q(2))$.}
	
	Let $X$ be a surface over a field $K$. The group $H^3_{\M}(X,\Q(2))$ has the following presentation. It is generated by sums of the form 
	$$\sum (C_i,f_i)$$
	where $C_i$ are curves on $X$ and $f_i$ functions on them satisfying the co-cycle condition 
	$$\sum \div(f_i)=0$$
	Equivalently these can be though of as codimensional  $2$ subvarieties $Z$ of $X \times \CP^1$ such that 
	$$Z\cdot (X \times \{0\} -X \times \{\infty\})=0$$
	 The cycle given by the sum of the graphs of $f_i$ in $X \times \CP^1$ has this property. 
	
	Relations are given by the Tame symbol of a pair of functions. If $f$ and $g$ are two functions on $X$, the  Tame symbol of the pair
	$$\tau(f,g)=\sum_{Z \in X^1} \left(Z,(-1)^{\ord_g(Z)\ord_f(Z)}\frac{f^{\ord_g(Z)}}{g^{\ord_f(Z)}}\right)$$
	satisfies the co-cycle condition and  such cycles are  defined to be $0$ is the group. Here $X^1$ is the set of irreducible codimensional one subvarieties. 
	
	One example of an element of this group is a cycle of the form $(C,a)$ where $a$ is a nonzero constant function and $C$ is a curve on $X$. This trivially satisfies the cocycle condition. More generally, if $L/K$ is a finite extension and $X_L=X \times_{\Spec(K)} \Spec(L)$ then one has a product map 
	$$\bigoplus_{L/K} H^2_{\M}(X_L,\Q(1)) \otimes H^1_{\M}(X_L,\Q(1)) \longrightarrow \bigoplus_{L/K} H^3_{\M}(X_L,\Q(2)) \stackrel{Nm_{L:K}}{\longrightarrow} H^3_{\M}(X,\Q(2))$$
		The image of this is the group of {\em decomposable cyles} $H^3_{\M}(X,\Q(2))_{dec}$. The group of {\em indecomposable cycles} is the quotient group
		$$H^3_{\M}(X,\Q(2))_{indec}=H^3_{\M}(X,\Q(2))/H^3_{\M}(X,\Q(2))_{dec}$$
		In general it is not clear how to construct non-trivial elements of this group.  
	
	A way of constructing cycles in this group is as follows. Since we use it in the main construction we label it a proposition. 
	
	\begin{prop}  Let $Q$ be a {\em nodal rational curve} on $X$ with node $P$. Let $\nu:\tilde{Q}\rightarrow Q$ be its  normalization in the blow up $\pi:\tilde{X} \longrightarrow X$ at $P$  The strict transform $\tilde{Q}$ meets the exceptional fibre $E_P$ at two points $P_1$ and $P_2$. Both $\tilde{Q}$ and $E_P$ are rational curves. Let $f_P$ be a function with $\div(f_P)=P_1-P_2$ on $\tilde{Q}$ and similarly let $g_P$ be a function with $\div(g_P)=P_2-P_1$ on $E_P$. Then 
		$$\tilde{Z}_P=(\tilde{Q},f_P)+(E_P,g_P)$$
		is an element of $H^3_{\M}(\tilde{X},\Q(2))$ and its pushforward $Z_P=\pi_*(\tilde{Z}_P)$ is an element of $H^3_{\M}(X,\Q(2))$.
		\label{rcconstruction}
		
		This is cycle is well defined up to decomposable cycles as the functions $f_P$ and $g_P$ are only defined up to scalars. To further determine the cycle we can choose a point on $\tilde{Q}$ and $E_P$ each and require the value of $f_P$ and $g_P$ to be $1$ a those points. 
	\end{prop}
A priori  it is not clear that if this cycle is indecomposable.  We will show that is the case in many instances. For this we need to use the localization sequence in motivic cohomology. 

\subsection{The localization sequence}
\label{localization}

If $\XX \rightarrow S$ is a family of varieties over a base $S$ with $X_{\eta}$ the generic fibre and $X_s$ the fibre over a closed subvariety $s$  then one has a localization sequence relating the three. 
$$ \cdots \rightarrow H^3_{\M}(X_{\eta},\Q(2) \stackrel{\partial}{\longrightarrow} \bigoplus_{s \in S^1} H^2_{\M}(X_s,\Q(1)) \longrightarrow H^4_{\M}(\XX,\Q(2)) \rightarrow \cdots $$
where $S^1$ denotes the set of irreducible closed  subvarieties of codimension $1$. The boundary map of an element 
$Z=\sum_i (C_{\eta,i},f_{\eta,i})$ is given as follows. Let $\CC$ denote the closure of $C_{\eta}$ in $\XX$. Then 
$$\partial(Z)=\sum_i \div_{\CC_i}(f_{\eta,i})$$
The cocycle condition $\sum_i \div_{C_{\eta,i}}(f_{\eta,i})=0$ implies that the `horizontal' divisors cancel out and only the components in some of the  fibres survive. 

Let $\XX$ be a family of $K3$ surfaces. To check for indecomposability one can use the boundary map $\partial$. Over the algebraic closure of the base field, a decomposable cycle in $H^3_{\M}(X_{\eta},\Q(2))$ is given by $(C_{\eta},a_{\eta})$ where $C_{\eta}$ is a curve on $X_{\eta}$ and $a_{\eta}$ a constant (that is, a function on the base $S$). 

The boundary map is particularly simple to compute in the case of a decomposable element 
$$\partial((C_{\eta},a_{\eta}))=\sum \ord_s(a_{\eta}) C_s$$
where $C_s$ is the restriction of the closure $\CC$ of $C_{\eta}$ to the fibre over $s$. Hence in the boundary of decomposable elements one can obtain only those cycles in the special fibres which are the restrictions of cycles in the generic fibre. In particular, if $Z_{\eta}$ is a motivic cycle and the boundary contains cycles which are {\em not} the restriction of cyces in the generic fibre then the motivic cycle is necessarily indecomposable. We will show that that is the case for our cycles.

An alternate way, as is done in Chen-Lewis \cite{chle}, Ma-Sato \cite{masa} or Sato \cite{sato}, is to compute the Beilinson regulator and show that, when evaluated against a particular transcendental $(1,1)$ form, this is non-zero.  In the case when the base is the ring of integers of a glabal field or a local field, the boundary map can be thought of as a non-Archimedean regulator map and the idea is essentially the same. The transcendental form is represented by a new algebraic cycle in the special fibre.  Evaluation against this form corresponds to intersecting with the new cycle.  That being non-zero implies that the component of the boundary in the direction of the new cycle is non-zero. This also fits in to the philosophy espoused by Manin \cite{mani} that the Archimedean component of an arithmetic variety can be viewed as the `special fibre at $\infty$'.

	\section{Construction of the cycles on the generic $K3$ surface of degree $2$.}

	We use Proposition \ref{rcconstruction}  to construct infinitely many distinct cycles in the generic  $K3$ surface of degree $2$. Recall that if $X$ is a $K3$ surface whose moduli point $s$ lies on $\KP_{Q}$, there is a rational curve $Q$ meeting $S_X$ at $3d$ double points. Assuming this moduli is non-empty, we construct a motivic cycle defined in a Zariski open set in the complement of this space. 
	
	To apply  Proposition \ref{rcconstruction} we need a suitable rational curve in the generic $K3$ surface. The idea is to `deform' the curve $Q$ which exists on the fibres over $\KP_Q$ to a rational curve which exists everywhere but meets the sextic at fewer points. 
	
	\subsection{A Theorem in Enumerative Geometry}

	There is a well known theorem in enumerative geometry \cite{zing} which states the following:
	
	\begin{thm} Let $N_d$ denote the number of rational curves of degree $d$ passing through $3d-1$ points in $\CP^2$ in general position. Then $N_d$ is  {\em finite} and {\em non-zero}.
		
	\end{thm}
	
	\begin{proof}  A proof of this can be found in the book by Dusa McDuff and Dieter Salamon \cite{mcsa}, page 213. \end{proof}
	
	The numbers $N_1=1$, $N_2=1$, $N_3=12$ are classical. The precise number $N_d$ is the celebrated theorem of Kontsevich-Manin \cite{koma} and Ruan-Tian \cite{ruti}. 
	
	Recall that $\KP_Q$ is the submoduli of the moduli of  $K3$ surfaces $X$ of degree $2$ where there is a rational curve $Q$ of degree $d$ which meets the sextic $S_X$ at points of even multiplicity. Since $Q$ is of degree $d$ it meets $S_X$ at $6d$ points. However, since they meet at either ordinary double points  of the sextic or points of even multiplicity, there are at most $3d$ distinct points of intersection. 
	
	If we discard one of the double points of $Q \cap S_X$, we determine $3d-1$ points. Hence the theorem above suggests that perhaps for {\em any} $S_X$ we can find a rational curve of degree $d$ passing through these $3d-1$ points. 
	
	In general it is not clear if  $\KP_Q$  is non-empty but we will give several examples when such submoduli exist. 
	
	We need some definitions. Let $\Sigma_6$ be the space of sextic curves in $\CP^2$ and let $\Sigma_6(n)$ be the subspace with $n$ nodes. A {\em singular type} $\Sigma_T$  \cite{yuzh}, Section 3, is an irreducible component of $\Sigma_6(n)$. A sextic in $\Sigma_T$ is said to be of type $T$. For $z \in \Sigma_T$ there are $n$ nodes $p_1(z),\dots,p_n(z)$ smoothly varying with $z$. 
	
	We have the following theorem, the proof of which was suggested to me by Tom Graber. The proof is easier as we assume that for some $X$ there exists a rational curve meeting $S_X$ at the correct number of points.  
	 
	 \begin{thm} \label{enumgeom}
	 	
	 	Consider pairs $(S,Q)$ where $S$ is a sextic in $\CP^2$  in $\Sigma_T$ of singular type $T$ with $n$ nodes and $Q$ is a rational curve of degree $d$. In general they will meet at $6d$ points. We make the following assumption: 
	 	
	 	{\bf Assumption}: For some $z_0$ in $\Sigma_T$, there exists a sextic $S_{z_0}$ of type $T$ and and a nodal rational curve $Q_{z_0}$ of degree $d$  such that $S_{z_0}$ and $Q_{z_0}$ meet at $ \leq 3d+1$   points of the following type 
	 	
	 	\begin{enumerate}
	 		\item $k$  nodes  of $S_{z_0}$, say $p_{i_1}(z_0),\dots,p_{i_k}(z_0)$. 
	 		\item $3d-1-k$ other points tangent to $S_{z_0}$ (which could coincide to be of even multiplicity). 
	 		\item $2$ other points (which could coincide and be of one  of the above types). 
	 	\end{enumerate}
	 	
	 	Then, for {\bf any} sextic $S_z$ corresponding to a point $z$ in $\Sigma_T$,
	 	there exists a nodal rational curves $Q_z$ of degree $d$ meeting the sextic at $3d-1$ 
	 	points of even multiplicity and two other points. There are two possibilities - either $3d-1-k$ points of tangency and passing through the $k$ nodes $p_{i_1}(z),\dots,p_{i_k}(z)$ or meeting the sextic at $3d-k$ points of tangency and passing thought $k-1$ nodes - and both exist.

	 \end{thm}
	 
	 \begin{proof}  
	 	
	 	Let $\M_d$ be the space of degree $d$ nodal rational curves in $\CP^2$ and let $\Sigma_T$ be the space of sextic curves of singular type $T$.  It is known that the dimension $\dim \M_d=3d-1$. Let $\M=\M_d(k)$ be the subspace of $\M_d$ of rational curves passing through $k$ points in general position. This is of dimension $3d-k-1$.  
	 	
	 	Let $\XX$ be the locus in $\M\times \Sigma_T$ parametrizing pairs $(Q_z,S_z)$ where $Q_z$ meets $S_z$ at $(3d-1)$ double points of which  $k$ are the nodes $p_{i_1}(z),\dots,p_{i_k}(z)$ of $S_z$ and two other points. $\XX$ is non-empty as by assumption $(S_{z_0},Q_{z_0})$ corresponds to a point on $\XX$. There is a morphism 
	 	$$\pi:\M \times \Sigma_T \longrightarrow \UH_{6d-2k}$$ 
	 	$$\pi(Q,S) \longrightarrow (Q \cap S)\backslash \{p_{i_1},\dots, p_{i_k}\}$$
	 	where $\UH_{6d-2k}$ is the Hilbert scheme of $6d-2k$ points in $\CP^2$. 
	 	
	 	Let $Z$ be the stratum of $\UH_{6d-2k}$ parametrizing subschemes of $3d-k-1$ double points $\cup$ $2$ reduced points. Then $\XX$ is just $\pi^{-1}(Z)$. The closure $\bar{Z}$ is an irreducible codimension $(3d-k-1)$ subscheme of $\UH_{6d-2k}$ as it is determined by choosing $3d-k+1$ points - which is $6d-2k+2$ dimensional - and then $3d-k-1$ tangents at $3d-1-k$ of those points - which makes it $9d-2k+1$ dimensional. Further $\UH_{6d-2k}$ is $12d-4k$ dimensional. Since $\UH_{6d-2k}$ is smooth   every irreducible component of $\pi^{-1}(\bar{Z})$ has codimension at most $(3d-k-1)$. 
	 	
	 	From our assumption we know that for some $z_0$ in $\Sigma_T$ there is a pair $(Q_{z_0},S_{z_0})$ in $\pi^{-1}(\bar{Z})$. Equivalently, the local dimension of $\XX \cap (\M \times \{S_{z_0}\})$ is $\geq 0$. Looking at $\pi$ near $(Q_{z_0},S_{z_0})$, since $\dim(\XX) \geq \dim (\Sigma_T)$, the map $\pi:\XX \rightarrow Z$ is flat (\cite{hart},Chapter III, Theorem 9.9) and  the fact that the local fibre dimension is $\geq 0$ at $(Q_{z_0},S_{z_0})$ implies all nearby fibre dimensions are $\geq 0$ . This is equivalent to the fact that there is such a  pair $(Q_z,S_z)$ for every $z$ in a Zariski open neighborhood of $z_0$. Since $\Sigma_T$ is irreducible, this holds for all pairs $(Q_z,S_z)$ with $z$ in $\Sigma_T$.

	 \end{proof}
	 
	 \begin{rem} In fact, we expect that the dimension is $0$ - namely there are a finite non-zero number of rational curves $Q$ for each $S$ of type $T$ when there is a special rational curve of the same type. There are instances when the special rational curve cannot exist - for instance, one cannot have a nodal cubic meeting six lines only at points of tangency.
	 \end{rem}

	 In general it is not clear if there exists pairs $(S_{t_0},Q_{t_0})$ of the type we require, though in several special cases they are known to exist. For instance if $k=0$ and $d=2$, so generically $S$ is  a smooth sextic  and $Q$ is a conic. We know that for the Kummer surface of a principally polarised Abelian surface there is a conic tangent to the sextic (which in this case is the union of six lines) at the six lines. In particular, at five lines. Hence from the theorem above, it says that in a Zariski neigbourhood of the moduli point of the Kummer surface  of an Abelian surface there exists a conic tangent to a sextic at five points. In fact, it is a theorem of Gathmann \cite{gath} (Corollary 3.6) that there are  $70956$ conics tangent to a sextic at five points
	 
	  If $S$ has nodes, for instance if $S$ is the union of six lines and has $15$ nodes, then the theorems of Birkenhake and Lange (\cite{biwi}, Theorems 7.1-7.4) show that  for  $\Delta$ of the form $2d^2-2k+7$ or $2d^2-2k+8$ there exists rational curves of degree $d$ passing through $k$ nodes and meeting the remaining lines at points of even multiplicity. In particular, there exists a rational curve of the type required passing through $k$ nodes and $3d-1-k$ points of even multiplicity.

	 If $S$ has $n$-nodes with $n\leq 8$ then  Yu and Zheng \cite{yuzh} show that the desingularization of the  double cover of $\CP^2$ branched at  $S$ in $\Sigma_T$ is a del Pezzo surface. So this argument extends to certain families of $K3$ double covers of del Pezzo surfaces.

	 A related case when there is a similar theorem is $\CP^1 \times \CP^1$. Here minimal desingularization of the double cover of $\CPP$ ramified at  a curve of degree $(4,4)$ is also a $K3$ surface. Here the analogous statement would be that if there existed a rational curve $Q$ of degree $(1,n)$ or $(n,1)$ meeting the curve at $2n+1$ double points then in a Zariski open set there exists such a curve.

	  In the special case that the curve is consists of two pairs of four parallel lines, the double cover is the Kummer surface of a product of elliptic curves, say $E_1$ and $E_2$. If $E_1$ and $E_2$ are isogenous by an isogeny of degree $n$ then image of the graph of the isogeny provides a example of the special curve \cite{sree-split}.

	 \subsection{Motivic cycles}

	 \subsubsection{Construction}

	  We now use Theorem \ref{enumgeom} to construct a motivic cycle.

	 \begin{const} Let $\KP_{Q}$ as above be the moduli of degree $2$ $K3$ surfaces $X$ for which there is a rational curve $Q$ of degree $d$ meeting the associated sextic $S_X$ at $3d$ points of even multiplicity. Let $T$ be a singular type and assume $\KP_Q \cap \Sigma_T$ is non-empty and that for $z_0$ in this intersection, $Q_{z_0}$ meets $S_{z_0}$ at $k$ of its nodes.   Then there exists a motivic cycle 
	 	$$Z_{Q_{z}} \in H^3_{\M}(X_z,\Q(2))$$
	 where $X_z$ is  the  $K3$ surface of degree $2$ corresponding to a $z$ in $V=\Sigma_T \backslash \KP_Q$.  
	 
	 This construction can be extended to get a cycle $Z_U$ defined in the motivic cohomology over a Zariski open set  $U$ lying over $V$ in an \'{e}tale cover of $\Sigma_T$. 
	 	\label{construction}
	 	\end{const}
	 	
	 \begin{proof} 
	 	Let $X_z$ denote the $K3$ surface corresponding to a point $z$. Let $S_z=S_{X_z}$ be the associated sextic. Let $\BX_z$ denote the surface which is obtained by blowing down the exceptional cycles of $X_z$ lying over the nodes of $S_z$ - so $\BX_z$ is a double cover of $\CP^2$ ramified at $S_z$. 
	 	
	 	If $z_0$ lies on $\KP_Q \cap \Sigma_T$, the curve $Q_{z_0}$ determines $3d$ points $Q \cap S_{z_0}$. From  Theorem \ref{enumgeom} above, for any $z \in \Sigma_T$, 
	 	there exists a rational curve $Q_z$ which meets $S_z$ at $3d-1$ points of the type determined by $Q_{z_0} \cap S_{z_0}$. $Q_z$ will meet $S_z$ at two other points $s_z$ and $t_z$ which are {\em not} points of even multiplicity. 
	 	
	 	Let $\pi:C_z \longrightarrow Q_z$ be the double cover of $Q_z$ in $\BX_z$. The normalization of $C_z$ is a double cover of $\CP^1$ ramified at two points - namely the points $s_z$ and $t_z$ and so is an {\em irreducible rational curve}. The curve $C_z$  has nodes at the points $C_z\cap S_z$ and  there are $3d-1$ such nodes. 
	 	
	 	Let $P_z$ be one of the nodes. If it is a node of $S_z$ as well then from Proposition \ref{rcconstruction}, using the function $f_{P_z}$ on the strict transform $\TC_z$ of $C_z$ along with the function $g_{P_z}$ on the exceptional fibre  $E_{P_z}$, we obtain a cycle 
	 			$$Z_{Q_z}=(\TC_{P_z},f_{P_z})+ (E_{P_z},g_{P_z}) \in H^3_{\M}(\BX_z,\Q(2))$$
	 				
	    If $P_z$ is not a node of $S_z$ then let $b:X_{P_z} \longrightarrow X_z$ be the blow up of $X_z$ at $P_z$. We can construct a function on the strict transform of $C_z$ in $X_{P_z}$ as above to obtain a cycle in $ H^3_{\M}(X_{P_z},\Q(2))$ and push that down using $b_*$ to obtain a cycle $Z_{Q_z} \in H^3_{\M}(X_z,\Q(2))$.
	    
	    We further determine $Z_{Q_z}$  by requiring the function $f_{P_z}$ satisfy
	    $$f_{P_z}(s_z)=1$$ 

	 	Monodromy around $\KP_Q$ might take one of the points in the strict transform of $C_z$  lying over $P_z$ to the other, so the function $f$ may not be defined over the generic fibre. To rectify this, we have to go to an \'{e}tale cover of the moduli space where all the sections are defined. This results in a cycle $Z_{Q_{U}}$ define in the fibres over a Zariski open set $U$ which meets $\KP_{Q}$ and one has sections $P_{U,1}$ and $P_{U,2}$ defined over $U$ which map to a node $P_U$. 
	 	
	\end{proof}

	For instance, consider the case of the conics meeting the sextic tangentially.  As remarked above, Gathmann \cite{gath} showed  there always exists a conic meeting the sextic tangentially at five points. This conic meets the sextic at two other points.  The normalisation of the double cover of this conic ramified at all the points of intersection is a double cover of $\CP^1$ ramified at two points - namely $t_z$ and $s_z$ - hence is an irreducible smooth rational curve. Using this we can construct a motivic cycle on the complement of the moduli of $K3$ surfaces where the conic meets the sextic at six points of tangency. 
	
	However, we do not have to use the theorem of Gathmann. We can instead use the fact that in the $\CP^2$  corresponding to the Kummer surface of a simple principally polarized Abelian surface there is a conic meeting the six lines tangentially. Hence at least in a Zariski open neighbourhood of Kummer surfaces of Abelian surfaces there exists a conic meeting the sextic tangentially at $5$ points. 
	
	In the special case when the sextic degenerates to a product of six lines - so the moduli space is $4$ dimensional. The $K3$ surface is obtained by blowing up the nodes of the double cover of $\CP^2$ ramified at these $6$ lines. There are $15$ nodes. The generic Picard number of the degree $2$ $K3$ surface $X$ corresponding to this type of a sextic is $16$. $15$ of those come from the exceptional fibres and one more coming from the hyperplane section. 
	
	The conic above will be tangent to $5$ of the six lines and will meet the remaining line at $2$ points. If these points coincide then the conic is tangent to the six lines and then one knows that corresponding $K3$ surface is the Kummer surface of an Abelian surface. And further, the Picard number is $17$ - one more cycle comes from a component of the double cover of the conic. 
	
	This is an example of the above situation - on the $4$ dimensional moduli of double covers of degenerate sextics, there is a motivic cycle supported in the fibres over the complement of the moduli of Kummer surfaces.  This is distinct from the cycles constructed in \cite{sree-simple} or \cite{masa}.

	In the construction we have made some choices. We chose a point to discard to obtain $3d-1$ points and we also chose the node $P_{U}$. A different choice of point to discard results in a different  rational curve $Q_{U}$ but the support of the boundary is the same. Another choice of $P_{U}$  will result in a different function $f_{P_{U}}$.  In the next section we  show that  all these choices will result in cycles with the same boundary. Interestingly, Sato \cite{sato} shows that in the case of products of elliptic curves,  the Archimedean regulator {\em does} depend on the choice of $P_{U}$. 
	
	As things stand, we do not know that the cycle is non zero. Our computation of the boundary will show that the cycle is generically indecomposable. In particular, that will show that it is generically non-trivial.

\subsubsection{Indecomposability}

 We now claim that $Z_{Q_U}$ in $H^3_{\M}(X_U,\Q(2))$ is in fact an indecomposable cycle.
 
 Recall that over $\KP_{Q}$ the conic $C_U$ gains a new node $R_Q$ since $s_U$ and $t_U$ coincide. As a result it breaks up in to two components. 
 
 There are cycles $\TC_{\KP_Q,1}$ and $\TC_{\KP_Q,2}$ in the universal $K3$ surface over $\KP_Q$ such that
 $$\TCC|_{\KP_Q}=\TC_{\KP_Q,1}+\TC_{\KP_Q,2}+ E_{R_Q}$$
 where $\TCC$ denotes the closure of $\TC_U$ in the universal $K3$ surface over the moduli space (or its blow up at $P_{U}$) and $E_{R_Q}$ the exceptional fibre over the node $R_Q$.

As explained in Section \ref{localization}, if a cycle is decomposable, then the boundary under the connecting homomorphism of the localization sequence is the restriction of a  generic cycle. If we show that the boundary is {\em not} the restriction of a generic cycle then it is is necessarily indecomposable.

	\begin{thm} The boundary of the cycle $Z_{Q_{U}}$ is supported in the fibres over  $\KP_Q$ and 
		$$\partial(Z_{Q_U})=\TC_{\KP_Q,1}-\TC_{\KP_Q,2}.$$
		The cycle $\TC_{\KP_Q,1}-\TC_{\KP_Q,2}$ is not the restriction of a cycle in the generic fibre  and hence the cycle $Z_{Q_U}$ is {\em indecomposable}
		
		\end{thm}

	\begin{proof}

	 Let $Z_{Q_U}=(\TC_U,f_{P_U})+(E_{P_U},g_{P_U})$ be the cycle in the generic fibre constructed in Theorem \ref{construction}. 
	$$\div(f_{P_U})=P_{U,1}-P_{U,2} $$
	where $P_{U,1}$ and $P_{U,2}$ are the two points lying over $P_U$.  We have further assumed that $f_{P_U}(s_U)=1$, where $s_U$ is one of the non-singular ramified points.

	 	 The double cover $\pi:X_U \rightarrow \CP^2_{K_U}$ induces an involution $\iota$ on $C_U$ which fixes the ramified points.  This involution lifts to the normalization  $\TC_U$. If $P_{U,1}$ and $P_{U,2}$ are two points lying over a singular ramified point $P_U$, then $\iota(P_{U,1})=P_{U,2}$. Further $\iota$ stabilizes the exceptional fibre: $\iota (E_{P_U})=E_{P_U}$ for any node $P_U$. 
	 	 
	 	 Let $f^{\iota}=f \circ \iota$. Then $f^{\iota}_{P_U}$ has divisor 
	 	 $$\div(f^{\iota}_{P_U})=P_{U,2}-P_{U,1}=-\div(f_{P_U})$$
	 	 hence $$f^{\iota}_{P_U}=\frac{c_U}{f_{P_U}}$$ for some constant (function on the base) $c_U$. Since we have further assumed that $f(s_U)=1$ and $\iota(s_U)=s_U$ we have $c_U \equiv 1$.

	 	 The closure  $\CC_Q$ of $\TC_U$ has two components over $\KP_Q$,  $\TC_{\KP_Q,1}$ and $\TC_{\KP_Q,2}$  and the involution  $\iota$ interchanges them. The closure of one of the points lying over $P_{U}$ will lie on $\TC_{\KP_Q,1}$ and the other on $\TC_{\KP_Q,2}$. Let $P_{\KP_Q,i}$ denote the closures of $P_{U,i}$. We can assume $P_{\KP_Q,i}$  lies on $\TC_{\KP_Q,i}$. 
	 	 
	 	 The closure of points $s_U$ and $t_U$ meet over $\KP_Q$ at a point $R_Q$ hence $\CC_Q$ has a node at $R_Q$. Let $\TCC_Q$ be the blow up of $\CC_Q$ at $R_Q$. This has fibre 
	 	 $$\TC_{\KP_Q,1}+\TC_{\KP_Q,2}+E_{R_Q}$$ 
	 	 over $\KP_Q$.  The involution stabilizes $E_{R_Q}$ as well. 
	 	 
	 	 The boundary map in the localization sequence is 
	 	 $$\partial(Z_{Q_U})=\div_{\TCC_Q}(f_{P_U})+\div_{E_{P_U}}(g_{P_U}).$$
	 	 The first term is 
	 	 $$\div_{\TCC_Q}(f_{P_U})=\UH + a_1\TC_{\KP_Q,1}+ a_2\TC_{\KP_Q,2} + a_3 E_{R_Q}$$
	 	 where $\UH$ is the closure of the `horizontal' divisor $\div(f_{P_U})$. We then recall that 
		 $$\div_{\TCC_Q}(f^{\iota}_{P_U})=\iota ( \div(f_{P_U})) \text{ and } f^{\iota}_{P_U}=\frac{1}{f_{P_U}}$$
		 Therefore, on one hand 
		 $$\div_{\TCC_Q}(f^{\iota}_{P_U})=\iota(\UH+ a_1\TC_{\KP_Q,1}+ a_2\TC_{\KP_Q,2} + a_3 E_{R_Q})=-{\mathcal H}+ a_1\TC_{\KP_Q,2}+ a_2\TC_{\KP_Q,1} + a_3 E_{R_Q}$$
		 and on the other 
		 $$\div_{\TCC_Q}(f^{\iota}_{P_U})=-\div_{\overline{\TC}_{P_U}}(f_{P_U})=-{\mathcal H}-a_1\TC_{\KP_Q,1}- a_2\TC_{\KP_Q,2}-a_3 E_{R_Q}.$$
		 This implies $a_1=-a_2$ and $a_3=-a_3$. Hence $a_3=0$. Let $a=a_1$. We have   
		 $$\div_{\TCC_Q}(f_{P_U})={\mathcal H}+ a(\TC_{\KP_Q,1} - \TC_{\KP_Q,2}).$$
		 To show $a \neq 0$ we suppose $a=0$. Then 
		 $$\div_{\TCC_Q}(f_{P_U})={\mathcal H}$$
		 Since $\div_{\TCC_Q} (f_{P_U})$ does not contain $\TC_{\KP_Q,1}$ it restricts to a function on $\TC_{\KP_Q,1}$ and so 
		 $$\deg(\div_{\TC_{\KP_Q,1}} (f_{P_U})|_{\TC_{\KP_Q,1}})=0.$$ 
		 However,  
		 $$\div_{\TC_{\KP_Q,1}} (f_{P_U})|_{\TC_{\KP_Q,1}})= {\mathcal H}.\TC_{\KP_Q,1}=P_{1,\KP_Q}$$
		 which  has degree $1$. This is a contradiction, hence $a \neq 0$.  
		 
		 Finally, $\div_{E_P}(g_P)=-{\mathcal H}$ as $E_{\KP_Q}$ remains irreducible when $z$ lies on $\KP_Q$. Hence 
		 $$\partial(Z_{Q_U})=a\left(\TC_{\KP_Q,1}- \TC_{\KP_Q,2}\right).$$
		 for some $a\neq 0$. In particular, since $\TC_{\KP_Q,1}- \TC_{\KP_Q,2}$ is not the restriction of a cycle in the generic fibre, the cycle $Z_Q$ is indecomposable.

To compute $a$ we do the following. Let $h$ be a function on the base with divisor $\div(h)=\KP_{Q}-\sum a_D D$ and such that $Supp(D) \cap \KP_Q=\emptyset$. Such a function exists as $H^1$ of the moduli spaces is $0$ so homological and rational equivalence are the same. Let $(\TC_U,h)$ be the decomposable cycle in $H^3_{\M}(X_U,\Q(2))$. Then 
$$\partial ((\TC_U,h))=\TC_{\KP_Q,1}+ \TC_{\KP_Q,2}+ E_{R_Q} - \sum a_D \TCC_{Q,D}$$
where $\TCC_{Q,D}$ is the $\TCC_Q|_D$. These fibres are irreducible and have no intersection with cycles in the fibres over $\KP_Q$. 

The function $h^{-a}f_{P_U}$ has divisor
$$\div_{\TCC_Q}(h^{-a}f_{P_U})=\UH + a\left(\TC_{\KP_Q,1}- \TC_{\KP_Q,2}\right) - a\left(\TC_{\KP_Q,1} +  \TC_{\KP_Q,2} + E_{R_Q} - \sum a_D \TCC_{Q,D}\right) $$
$$=\UH -2a \TC_{\KP_Q,2}-a E_{R_Q} + a \sum a_D \TCC_{Q,D} $$
so does not have $\TC_{\KP_Q,1}$ in its support. We can therefore restrict the function $h^{-a}f_{P_U}$ to it and 
$$\deg(\div(h^{-a}f_{P_U})|_{\TC_{\KP_Q,1}})=0$$
On the other hand one has 
$$\deg(\div(h^{-a}f_{P_U})|_{\TC_{\KP_Q,1}})=(\div_{\TCC_Q}(h^{-a}f_{P_U}),\TC_{\KP_Q,1})$$
Since $(\UH,\TC_{\KP_Q,1})=1$, $(\TC_{\KP_Q,2},\TC_{\KP_Q,1})=0$, $(E_{R_Q},\TC_{\KP_Q,1})=1$ and $ (\TCC_{Q,D},\TC_{\KP_Q,1})=0$ for all $D$, this implies  
$$0=1-a$$
Therefore $a=1$. 

\end{proof}

	\section{Applications}
	
	There are some immediate consequences of the existence of such indecomposable motivic cycles. In the case of the base being a local ring, Spiess \cite{spie} shows that it has some consequences for torsion in codimension $2$. Here we give a rough outline of possible other applications. 
	
	\subsection{Mumford's conjecture.} Mumford conjectured that there are infinitely many rational cures on a $K3$ surface. Mori and Mukai \cite{momu} showed this in the complex case  and Bogomolov, Hassett and Tschinkel \cite{BHT} showed this in the case of mixed characteristics using a delicate deformation argument deforming rational curves in special fibres to the generic fibre. 
	
	A byproduct of the construction above is that there are infinitely many rational curves on the general degree $2$ $K3$ surface.  The motivic cycles we have are made up of rational curves. Since the motivic cycles have boundaries on distinct moduli, the corresponding rational curves are distinct.

	\subsection{Hodge-$\D$-conjecture for $K3$ surfaces over a local field}  
	
	Let $X$ be a variety over a local field $K$ with residue field $k$. Let $X_k$ denote the special fibre of a regular proper model $\XX$ over $\OO_K$.  The Hodge $\D$-conjecture asserts that the map 
	$$H^3_{\M}(X,\Q(2)) \stackrel{\partial}{\longrightarrow} H^2_{\M}(\XX_k,\Q(1))$$
	is surjective. This is equivalent to showing that for any cycle in the Neron-Severi group there is a motivic cycle bounding it. 
	
	In several cases, if $X$ is a surface, the rank of the Neron-Severi of the special fibre is greater than that of the general fibre. In these cases, For instance, in the case of a product of elliptic curves without $CM$, in the special fibre the Frobenius induces complex multiplication and the graph of the Frobenius is a new cycle. 
	
	In some instances, the enumerative geometry argument might work directly in mixed characteristic (at least for odd primes). For instance, the statement that there is a conic tangent to  $5$ lines  general position is purely algebraic and extends to the mixed characteristic case.

	 \subsection{Relations with modular forms}
	 
	 The Heegner divisors on the moduli space of Abelian surfaces are intimately connected with modular forms, as in the case of the theorems of Gross-Zagier \cite{grza}, Hirzebruch-Zagier \cite{hiza}, van der Geer \cite{vdg} and Hermann \cite{herm}.
	 
	 The theorems state that Heegner divisors give rise to coefficients of modular forms. Borcherds \cite{borc} showed that all these theorems can be proved in a similar manner. The idea is that to show a certain power series is a modular form of a certain weight $k$, it suffies to show its coefficients satisfy relations between coefficients of modular forms of that weight. The space of `relations between coefficients of modular forms of weight $k$' is the space of weakly holomorphic modular forms of weight $2-k$. 
	 
	 So an approach to this theorem is to relate weakly holomorphic modular forms of weight $2-k$ with `relations between Heegner divisors'. Borcherds did so by showing that his lift of weakly holomorphic modular forms gave rise to functions with divisoral support on Heegner divisors - or in other words a relation of rational equivalence among Heegner divisors. 
	 
	 There are similar theorems expected  for higher codimension  Heegner cycles in the universal families over these Shimura varieties. Gross and Zagier \cite{grza} that Heegner divisors on a modular curve are related to modular forms of weight $\frac{3}{2}$. Zhang \cite{zhan} showed that Heegner cycles on self product of the universal family are related to higher weight modular forms. For instance if one consideres the threefold which is compactification of self-product the universal elliptic curve, then the Heegner cycles are related to coefficients of weight $\frac{5}{2}$.
	 
	 It it intriguing to speculate if these can be proved along the lines of Borcherds thereom. The modular form formalism extends without any problem. Relations among Heegner cycles are given by elements of the motivic cohomology group so the question becomes `Can one construct motivic cycles from weakly holomorphic modular forms?'. Borcherds theorem is an example of this as functions are elements of the motivic cohomology group $H^1_{\M}(S_{\eta},\Q(1))$ where $S_{\eta}$ is generic point of the modular variety. 
	 
	 Borcherds lifts are functions, so one cannot expect it to directly work for higher codimensions. However, given a motivic cycle in the generic fibre of the the universal family one can consider its regulator. This is a current on certain forms and for a suitable choice of form varying over the family computing the regulator on it gives a function on the base. Hence one can ask whether a certain Borcherds lift of weakly holomorphic forms is the regulator of a motivic cycle. In \cite{sree-simple}
	  and \cite{sree-split} we provide some evidence of this for codimension $2$. 
	 
	 In this paper what we construct are relations among codimension two cycles in the universal family of $K3$ surfaces of degree $2$ and one might expect that once again there is a relation between the regulator of these cycles and Borcherds lifts of modular forms.

	 \bibliographystyle{alpha}
	 \bibliography{AlgebraicCycles.bib}

\begin{thebibliography}{BHT11}

\bibitem[BHT11]{BHT}
Fedor Bogomolov, Brendan Hassett, and Yuri Tschinkel.
\newblock Constructing rational curves on {K}3 surfaces.
\newblock {\em Duke Math. J.}, 157(3):535--550, 2011.

\bibitem[Bor99]{borc}
Richard~E. Borcherds.
\newblock The {G}ross-{K}ohnen-{Z}agier theorem in higher dimensions.
\newblock {\em Duke Math. J.}, 97(2):219--233, 1999.

\bibitem[Bru08]{brun123}
Jan~Hendrik Bruinier.
\newblock Hilbert modular forms and their applications.
\newblock In {\em The 1-2-3 of modular forms}, Universitext, pages 105--179.
  Springer, Berlin, 2008.

\bibitem[BW03]{biwi}
Christina Birkenhake and Hannes Wilhelm.
\newblock Humbert surfaces and the {K}ummer plane.
\newblock {\em Trans. Amer. Math. Soc.}, 355(5):1819--1841, 2003.

\bibitem[CL05]{chle}
Xi~Chen and James~D. Lewis.
\newblock The {H}odge-{$ D$}-conjecture for {$K3$} and abelian surfaces.
\newblock {\em J. Algebraic Geom.}, 14(2):213--240, 2005.

\bibitem[CM17]{clma}
Adrian Clingher and Andreas Malmendier.
\newblock On the geometry of (1,2)-polarized kummer surfaces, 2017.

\bibitem[Gat05]{gath}
Andreas Gathmann.
\newblock The number of plane conics that are five-fold tangent to a given
  curve.
\newblock {\em Compos. Math.}, 141(2):487--501, 2005.

\bibitem[GZ86]{grza}
Benedict~H. Gross and Don~B. Zagier.
\newblock Heegner points and derivatives of {$L$}-series.
\newblock {\em Invent. Math.}, 84(2):225--320, 1986.

\bibitem[Har77]{hart}
Robin Hartshorne.
\newblock {\em Algebraic geometry}, volume No. 52 of {\em Graduate Texts in
  Mathematics}.
\newblock Springer-Verlag, New York-Heidelberg, 1977.

\bibitem[Her95]{herm}
C.~F. Hermann.
\newblock Some modular varieties related to {${\bf P}^4$}.
\newblock In {\em Abelian varieties ({E}gloffstein, 1993)}, pages 105--129. de
  Gruyter, Berlin, 1995.

\bibitem[HZ76]{hiza}
F.~Hirzebruch and D.~Zagier.
\newblock Intersection numbers of curves on {H}ilbert modular surfaces and
  modular forms of {N}ebentypus.
\newblock {\em Invent. Math.}, 36:57--113, 1976.

\bibitem[Ker13]{kerr}
Matt Kerr.
\newblock {$K^{\rm ind}_1$} of elliptically fibered {$K3$} surfaces: a tale of
  two cycles.
\newblock In {\em Arithmetic and geometry of {K}3 surfaces and {C}alabi-{Y}au
  threefolds}, volume~67 of {\em Fields Inst. Commun.}, pages 387--409.
  Springer, New York, 2013.

\bibitem[KM94]{koma}
M.~Kontsevich and Yu. Manin.
\newblock Gromov-{W}itten classes, quantum cohomology, and enumerative
  geometry.
\newblock {\em Comm. Math. Phys.}, 164(3):525--562, 1994.

\bibitem[Man91]{mani}
Yu~I. Manin.
\newblock Three-dimensional hyperbolic geometry as ...-adic arakelov geometry.
\newblock {\em Inventiones mathematicae}, 104(2):223--244, 1991.

\bibitem[May72]{maye}
Alan~L. Mayer.
\newblock Families of {$K-3$} surfaces.
\newblock {\em Nagoya Math. J.}, 48:1--17, 1972.

\bibitem[Mel08]{mell}
Anton Mellit.
\newblock Higher green's functions for modular forms, 2008.

\bibitem[Mil92]{mild}
Stephen J.~M. Mildenhall.
\newblock Cycles in a product of elliptic curves, and a group analogous to the
  class group.
\newblock {\em Duke Math. J.}, 67(2):387--406, 1992.

\bibitem[MM83]{momu}
Shigefumi Mori and Shigeru Mukai.
\newblock The uniruledness of the moduli space of curves of genus {$11$}.
\newblock In {\em Algebraic geometry ({T}okyo/{K}yoto, 1982)}, volume 1016 of
  {\em Lecture Notes in Math.}, pages 334--353. Springer, Berlin, 1983.

\bibitem[MS12]{mcsa}
Dusa McDuff and Dietmar Salamon.
\newblock {\em {$J$}-holomorphic curves and symplectic topology}, volume~52 of
  {\em American Mathematical Society Colloquium Publications}.
\newblock American Mathematical Society, Providence, RI, second edition, 2012.

\bibitem[MS23]{masa}
Shouhei Ma and Ken Sato.
\newblock Higher chow cycles on some k3 surfaces with involution, 2023.

\bibitem[Pet15]{pete}
Arie Peterson.
\newblock Modular forms on the moduli space of polarised k3 surfaces, 2015.

\bibitem[RT95]{ruti}
Yongbin Ruan and Gang Tian.
\newblock A mathematical theory of quantum cohomology.
\newblock {\em J. Differential Geom.}, 42(2):259--367, 1995.

\bibitem[Sat23]{sato}
Ken Sato.
\newblock A group action on higher chow cycles on a family of kummer surfaces.
\newblock 2023.

\bibitem[Sha80]{shah}
Jayant Shah.
\newblock A complete moduli space for {$K3$} surfaces of degree {$2$}.
\newblock {\em Ann. of Math. (2)}, 112(3):485--510, 1980.

\bibitem[Spi99]{spie}
Michael Spiess.
\newblock On indecomposable elements of {$K_1$} of a product of elliptic
  curves.
\newblock {\em $K$-Theory}, 17(4):363--383, 1999.

\bibitem[Sre08]{sree2008}
Ramesh Sreekantan.
\newblock A non-{A}rchimedean analogue of the {H}odge-{${\mathcal
  D}$}-conjecture for products of elliptic curves.
\newblock {\em J. Algebraic Geom.}, 17(4):781--798, 2008.

\bibitem[Sre14]{sree2014}
Ramesh Sreekantan.
\newblock Higher {C}how cycles on {A}belian surfaces and a non-{A}rchimedean
  analogue of the {H}odge-{$D$}-conjecture.
\newblock {\em Compos. Math.}, 150(4):691--711, 2014.

\bibitem[Sre22]{sree-simple}
Ramesh Sreekantan.
\newblock Algebraic cycles and values of green's functions, 2022.

\bibitem[Sre24]{sree-split}
Ramesh Sreekantan.
\newblock Algebraic cycles and values of green's functions 1- products of
  elliptic curves, 2024.

\bibitem[vdG82]{vdg}
G.~van~der Geer.
\newblock On the geometry of a {S}iegel modular threefold.
\newblock {\em Math. Ann.}, 260(3):317--350, 1982.

\bibitem[YZ23]{yuzh}
Chenglong Yu and Zhiwei Zheng.
\newblock Moduli of nodal sextic curves via periods of {$K3$} surfaces.
\newblock {\em Adv. Math.}, 430:Paper No. 109219, 20, 2023.

\bibitem[Zha97]{zhan}
Shouwu Zhang.
\newblock Heights of {H}eegner cycles and derivatives of {$L$}-series.
\newblock {\em Invent. Math.}, 130(1):99--152, 1997.

\bibitem[Zin05]{zing}
Alexsey Zinger.
\newblock Counting plane rational curves: Old and new approaches.
\newblock {\em Math ArXiv - math/0507105}, 2005.

\end{thebibliography}

\end{document}